\documentclass[11pt]{amsart}
\usepackage[T1]{fontenc}
\usepackage[utf8]{inputenc}
\usepackage[margin=1in]{geometry}
\usepackage{amsmath,amssymb,amsthm,booktabs}
\usepackage{microtype}
\usepackage{needspace}
\usepackage{tikz}
\usepackage[hidelinks]{hyperref}

\newtheorem{theorem}{Theorem}[section]
\newtheorem{proposition}[theorem]{Proposition}
\newtheorem{corollary}[theorem]{Corollary}
\theoremstyle{remark}
\newtheorem{remark}[theorem]{Remark}
\newcommand{\R}{\mathbb R}
\newcommand{\Q}{\mathbb Q}
\newcommand{\PP}{\mathbb P}
\newcommand{\A}{\mathcal A}
\newcommand{\cO}{\mathcal O}
\newcommand{\cyc}{\operatorname{cyc}}

\title[A sporadic simplicial arrangement of 51 lines]{A sporadic simplicial arrangement of 51 lines}
\author{Tom\'as S. R. Silva}
\address{Instituto de Matem\'atica, Estat\'istica e Computa\c{c}\~ao Cient\'ifica (IMECC), Universidade Estadual de Campinas (UNICAMP), 13083-859, Campinas, S\~ao Paulo, Brazil}
\email{tomas@ime.unicamp.br}
\urladdr{www.ime.unicamp.br/~tomas}
\date{\today}

\hypersetup{
  pdftitle={A sporadic simplicial arrangement of 51 lines},
  pdfauthor={Tomás S. R. Silva},
  pdfsubject={A counterexample to the 37-line bound for sporadic simplicial arrangements}
}

\begin{document}
\begin{abstract}
We construct a sporadic simplicial arrangement of $51$ lines in the real projective plane. It is obtained from the icosahedral arrangement $\mathcal{A}(21,2)$ by adjoining the $30$ images of the line $y+\sqrt{\varphi}\,z=0$ under the rotation group of the icosahedron, where $\varphi=(1+\sqrt5)/2$. Its intersection multiplicities are $t_2=120$, $t_3=190$, $t_5=36$, and $t_6=15$, with all other $t_r$ zero. These counts give equality in Melchior's inequality and hence prove simpliciality. The arrangement is irreducible, belongs to neither of the two known infinite families of irreducible simplicial arrangements, and is not supersolvable. It therefore disproves the $37$-line bound for sporadic arrangements conjectured by Cuntz and M\"ucksch. We give explicit coordinates, a finite incidence certificate, a geometric explanation of the parameter $\sqrt{\varphi}$, and a standalone exact-arithmetic verifier.
\end{abstract}
\maketitle

\section{Introduction}

A finite arrangement of lines in $\PP^2(\R)$ is \emph{simplicial} if every component of its complement is a triangle. Besides the near-pencils, the classical constructions give two infinite families: $\mathcal R(1)$, obtained from the sides and symmetry axes of regular polygons, and $\mathcal R(2)$, obtained by adding the line at infinity to the even-polygon examples. Arrangements outside these families are called \emph{sporadic}.

The classification of simplicial arrangements has been developed through geometric constructions, catalogues, and computational enumeration. Gr\"unbaum's catalogue~\cite{Gru09} collected the classical examples; subsequent work of Cuntz~\cite{Cun12,Cun22} enlarged the list and investigated arrangements with prescribed numbers of lines. The database accompanying Cuntz, Elia, and Labb\'e~\cite{CEL22} provides coordinates and combinatorial data for these examples. The largest sporadic arrangements in these lists have $37$ lines.

The expectation that $37$ is an upper bound goes back to Gr\"unbaum~\cite[Conjecture~2.1]{Gru74}. To state its later formulations, we pass from projective lines to the corresponding planes through the origin in $\R^3$; rank, irreducibility, and the characteristic polynomial refer to this associated central arrangement. In this language, Cuntz and M\"ucksch conjectured that every irreducible simplicial arrangement of rank three with more than $37$ hyperplanes belongs to $\mathcal R(1)\cup\mathcal R(2)$~\cite[Conjecture~1.1]{CM19}. Their classification theorem shows that every irreducible supersolvable simplicial arrangement of rank three is lattice equivalent to a member of these two families~\cite[Theorem~1.2]{CM19}. They therefore gave an equivalent formulation of the conjecture: every irreducible simplicial arrangement of rank three with more than $37$ hyperplanes is supersolvable~\cite[Conjecture~1.3]{CM19}.

We construct a counterexample to both conjectures. Set
\[
 \varphi=\frac{1+\sqrt5}{2},\qquad s=\sqrt{\varphi}>0,
 \qquad K=\Q(s),\qquad s^4-s^2-1=0.
\]
Let $I\cong A_5$ (the alternating group on $5$ symbols) be the rotation group of the regular icosahedron with vertices
\[
 (0,\pm\varphi,\pm1)
 \quad\text{and their cyclic permutations}.
\]
We write $[a:b:c]$ for the line $ax+by+cz=0$, and $(x:y:z)$ for a point. Since the elements of $I$ are orthogonal, their action on line normals is the usual matrix action.

\Needspace{18\baselineskip}
\begin{theorem}\label{thm:main}
The arrangement
\[
 \A_{51}=I\cdot[1:0:0]\ \cup\
          I\cdot[0:\varphi:1]\ \cup\
          I\cdot[0:1:s]
\]
consists of $51$ distinct real lines and is simplicial and irreducible. Its nonzero intersection numbers are
\[
 (t_2,t_3,t_5,t_6)=(120,190,36,15),
\]
where $t_r$ counts points incident with exactly $r$ lines. It has $361$ vertices, $1080$ edges, and $720$ triangular chambers in $\PP^2(\R)$. Its central characteristic polynomial is
\[
 \chi_{\A_{51}}(T)=(T-1)(T^2-50T+669).
\]
In particular, $\A_{51}$ is sporadic and is not supersolvable.
\end{theorem}

The construction adds one orbit of $30$ lines to the known arrangement $\mathcal{A}(21,2)$. Section~\ref{sec:construction} gives all $51$ normals in a compact form. Section~\ref{sec:certificate} proves the theorem using twelve point orbits and the pair-count identity. Section~\ref{sec:search} describes the computational search that led to the example. Section~\ref{sec:geometry} explains a triple concurrence that selects $s=\sqrt{\varphi}$. The proof requires only arithmetic in $K$; accompanying Python code reproduces every intersection exactly.

\section{The construction}\label{sec:construction}

The group $I$ has order $60$. For explicit computations it is generated by
\[
 \kappa=\begin{pmatrix}0&0&1\\1&0&0\\0&1&0\end{pmatrix},
 \qquad
 \eta=\begin{pmatrix}1&0&0\\0&-1&0\\0&0&-1\end{pmatrix},
 \qquad
 \rho=\frac12\begin{pmatrix}
 1&-\varphi&\varphi^{-1}\\
 \varphi&\varphi^{-1}&-1\\
 \varphi^{-1}&1&\varphi
 \end{pmatrix}.
\]
These matrices have orders $3$, $2$, and $5$, respectively, and preserve the twelve vertices of the icosahedron. Define
\[
 \cO_{15}=I\cdot[1:0:0],\qquad
 \cO_6=I\cdot[0:\varphi:1],\qquad
 \cO_{30}=I\cdot[0:1:s].
\]
The first orbit consists of the $15$ mirrors of the reflection arrangement of type $H_3$; the second consists of the six lines polar to the fivefold rotation axes. Their union is the classical simplicial arrangement $\mathcal{A}(21,2)$ of~\cite{Gru09}
\[
 B=\cO_{15}\cup\cO_6=\mathcal{A}(21,2).
\]
Here $\mathcal{A}(n,k)$ denotes a catalogue label: $n$ is the number of lines and $k$ distinguishes examples with that number (see~\cite{Gru09}). The notation $\A_{51}$ is used for the present construction without assigning it a catalogue index.

The following formulas specify all the lines without requiring group enumeration. Write $\cyc(a,b,c)$ for the set of lines represented by the three cyclic permutations of $(a,b,c)$. Then
\begin{align}
 \cO_{15}
 &=\cyc(1,0,0)\ \cup\
   \bigcup_{\epsilon,\epsilon'=\pm1}
       \cyc(1,\epsilon\varphi,\epsilon'\varphi^{-1}),\label{eq:15}\\
 \cO_6
 &=\bigcup_{\epsilon=\pm1}\cyc(0,\varphi,\epsilon),\label{eq:6}\\
 \cO_{30}
 &=\bigcup_{\epsilon=\pm1}\cyc(0,1,\epsilon s)\ \cup\
   \bigcup_{\epsilon,\epsilon',\delta=\pm1}
       \cyc\bigl(1,\epsilon(1+\delta s),
                    \epsilon'(1+\delta\varphi^{-1}s)\bigr).\label{eq:30}
\end{align}
All signs in the unions vary independently; the same sign $\delta$ occurs in both coordinates in~\eqref{eq:30}. The formulas contain $3+12$, $6$, and $6+24$ distinct projective normals, respectively, and the three sets are disjoint. Their equality with the stated $I$-orbits can be checked by applying the generators. Thus
$\A_{51}=\cO_{15}\cup\cO_6\cup\cO_{30}$ is an explicit real arrangement defined over $K$.

\begin{figure}[htbp]
\centering
\input{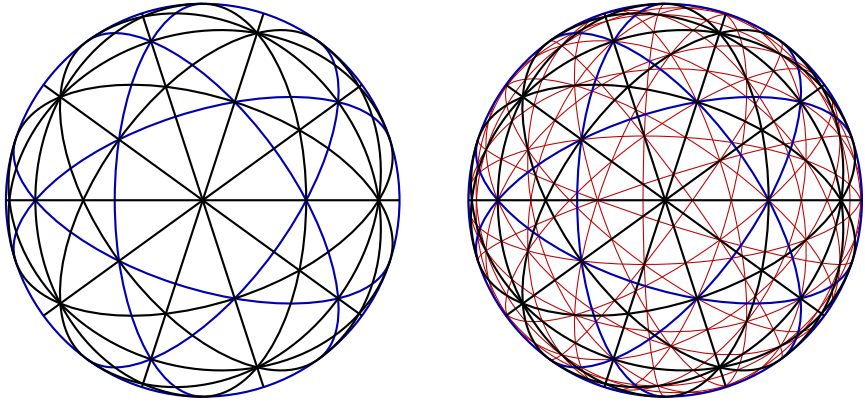}
\caption{The arrangements $B=\mathcal{A}(21,2)$ (left) and $\A_{51}$ (right), represented by great circles on a hemisphere viewed along a fivefold axis. Antipodal boundary points are identified. The $\cO_{15}$ lines are in black, $\cO_6$ in blue, and $\cO_{30}$ in red.}
\label{fig:arrangements}
\end{figure}

\section{Incidences and proof of simpliciality}\label{sec:certificate}

For an essential projective line arrangement, let $f_0,f_1,f_2$ denote the numbers of vertices, edges, and chambers. Euler's formula and edge--chamber incidence counting give
\[
 f_0=\sum_{r\ge2}t_r,\qquad
 f_1=\sum_{r\ge2}r\,t_r,\qquad
 f_2=1-f_0+f_1,\qquad 3f_2\le2f_1.
\]
The last inequality is an equality exactly when every chamber is triangular. Rearranging yields Melchior's inequality~\cite{Mel41},
\begin{equation}\label{eq:melchior}
 \Delta(\A):=t_2-3-\sum_{r\ge4}(r-3)t_r\ge0,
 \qquad
 \A\text{ is simplicial}\ \Longleftrightarrow\ \Delta(\A)=0.
\end{equation}

\begin{proposition}\label{prop:certificate}
The intersection points of $\A_{51}$ form the twelve $I$-orbits in Table~\ref{tab:points}. If a point has profile $(a,b,c)$, it lies on exactly $a$ lines of $\cO_{15}$, $b$ lines of $\cO_6$, and $c$ lines of $\cO_{30}$.
\end{proposition}

\begin{table}[htbp]
\centering
\small
\renewcommand{\arraystretch}{1.12}
\begin{tabular}{@{}lrrc@{}}
\toprule
Representative point & Orbit size & Multiplicity & Profile\\
\midrule
$(0:0:1)$                         &15&6&$(2,2,2)$\\
$(0:1:\varphi^{-1})$              &6 &5&$(5,0,0)$\\
$(0:1:-s)$                       &30&5&$(1,0,4)$\\
$(0:1:-1+\varphi^{-1}s)$          &30&3&$(1,0,2)$\\
$(0:1:-1-\varphi)$               &10&3&$(3,0,0)$\\
$(0:1:-1-\varphi^{-1}s)$          &30&3&$(1,0,2)$\\
$(0:1:-\varphi-\varphi s)$        &30&3&$(1,0,2)$\\
$(0:1:-\varphi+\varphi s)$        &30&3&$(1,0,2)$\\
$(\varphi:s:-1)$                 &60&3&$(0,1,2)$\\
$(0:1:\varphi^{-1}s)$             &30&2&$(1,0,1)$\\
$(0:1:-\varphi)$                 &30&2&$(1,1,0)$\\
$(1:-s^3:s^3)$                   &60&2&$(0,0,2)$\\
\bottomrule
\end{tabular}
\caption{A finite incidence certificate for $\A_{51}$. Each row denotes one full orbit under the $60$ rotations of $I$. The listed orbits are pairwise disjoint and contain $361$ points in total.}
\label{tab:points}
\end{table}

\begin{proof}
All entries are verified in $K=\Q[s]/(s^4-s^2-1)$. Apply the group generators to each representative, identifying proportional triples, to obtain its orbit and check that the twelve resulting sets are disjoint. Substitution of the representatives into the line equations~\eqref{eq:15}--\eqref{eq:30} gives the profiles in the last column. These operations involve only rational arithmetic and reduction by $s^4=s^2+1$; they are also implemented in the accompanying verifier.

For example, $(0:0:1)$ lies on the two coordinate mirrors $x=0,y=0$, the two lines $[\varphi:\pm1:0]$ of $\cO_6$, and the two lines $[1:\pm s:0]$ of $\cO_{30}$. Its profile is therefore $(2,2,2)$. The point $(\varphi:s:-1)$ lies on $[1:0:\varphi]$, $[0:1:s]$, and $[1:-s:0]$, giving the profile $(0,1,2)$; direct substitution excludes any further incident line.

Summing the orbit sizes at each multiplicity gives
\[
 t_2=120,\qquad t_3=190,\qquad t_5=36,\qquad t_6=15.
\]
To see that no intersection has been omitted, count unordered pairs of lines. The listed points account for
\begin{equation}\label{eq:pairs}
 120\binom22+190\binom32+36\binom52+15\binom62
 =1275=\binom{51}{2}.
\end{equation}
Every pair of distinct projective lines has exactly one intersection, so the table is complete.
\end{proof}

\begin{proof}[Proof of Theorem~\ref{thm:main}]
The coordinate formulas show that there are $51$ distinct lines, and the coordinate mirrors show that the arrangement has rank three. Proposition~\ref{prop:certificate} gives
\[
 \Delta(\A_{51})=120-3-2\cdot36-3\cdot15=0.
\]
Hence $\A_{51}$ is simplicial by~\eqref{eq:melchior}. The formulas for $f_i$ give
\[
 (f_0,f_1,f_2)=(361,1080,720).
\]
The associated arrangement on the sphere has $1440$ triangular chambers.

A reducible essential central arrangement of rank three is a product of arrangements of ranks one and two; its projectivization is a near-pencil. Such an arrangement of $51$ lines has a point of multiplicity $50$, whereas $\A_{51}$ has maximum multiplicity $6$. Thus $\A_{51}$ is irreducible. The families $\mathcal R(1)$ and $\mathcal R(2)$ have, respectively, $2n$ lines and $4m+1$ lines~\cite{CM19}. Since $51$ is odd and $51\equiv3\pmod4$, $\A_{51}$ belongs to neither family and is sporadic.

For the central characteristic polynomial, the coefficient of $T$ is
\[
 b_2=\sum_{r\ge2}(r-1)t_r
 =120+380+144+75=719.
\]
Using $\chi_{\A_{51}}(1)=0$ gives
\[
 \chi_{\A_{51}}(T)=T^3-51T^2+719T-669
                 =(T-1)(T^2-50T+669).
\]
Since $\A_{51}$ is irreducible and simplicial and its cardinality excludes lattice equivalence to any member of $\mathcal R(1)\cup\mathcal R(2)$, the classification of Cuntz and M\"ucksch~\cite[Theorem~1.2]{CM19} already implies that it is not supersolvable. The characteristic polynomial provides an independent check: its quadratic factor has discriminant $-176$, whereas supersolvable arrangements have characteristic polynomials that split into integer linear factors~\cite[Theorem~2.40]{CM19}.
\end{proof}

\begin{corollary}
The $37$-line bound for sporadic real simplicial arrangements is false.
\end{corollary}

The example does not settle whether there are only finitely many sporadic simplicial arrangements, or whether some larger uniform bound on their numbers of lines exists.

\section{How the arrangement was found}\label{sec:search}

The arrangement was discovered by a symmetry-restricted variant of Cuntz's greedy line-replacement algorithm~\cite{Cun22}. Its basic move replaces a line by the line through two existing intersection points, using a combinatorial measure of simpliciality to select changes. Searching among unions of symmetry orbits has precedents in the work of Cuntz--Geis~\cite{CG15} and in Cuntz's orbit-search algorithm~\cite{Cun18}. Our implementation combined these ideas with a protected known arrangement and computations over prescribed real number fields.

The successful run formed part of a search over quadratic extensions of the fields of known simplicial arrangements. It kept $B=\mathcal A(21,2)$ fixed, imposed invariance under the icosahedral rotation group $I$, and worked over $K=\Q(s)$, where $s=\sqrt{\varphi}>0$. Initial configurations were obtained by adjoining $I$-orbits of randomly chosen lines with small coefficients in $K$. The total number of lines was constrained to lie between $28$ and $51$.

A trial move replaced an orbit disjoint from $B$ by the complete $I$-orbit of a line joining two distinct intersection points of the current arrangement, subject to the cardinality constraint. Every configuration therefore retained $B$ and the prescribed symmetry. The objective was the simpliciality defect $\Delta$ in~\eqref{eq:melchior}. Moves decreasing $\Delta$ were accepted, with a bounded number of neutral moves allowed and random restarts when progress stalled. Modular computations screened proposals, and the resulting candidate was independently verified in exact arithmetic over $K$. Working directly with real realizations avoided the subsequent reconstruction of a characteristic-zero realization from a finite-field incidence structure.

The search produced
\[
 B\cup I\cdot[0:1:s]=\A_{51}.
\]
The incidence certificate in Section~\ref{sec:certificate} and the verifier described in Section~\ref{sec:verification} establish its properties independently of the search. The concurrence in Section~\ref{sec:geometry} provides a subsequent geometric explanation for the parameter $s$.

\section{Why the parameter \texorpdfstring{$\sqrt{\varphi}$}{sqrt(phi)} appears}\label{sec:geometry}

Consider the pencil $\ell_t=[0:1:t]$ through $P=(1:0:0)$, with $\ell_\infty=[0:0:1]$. For finite $t$, the two orbit lines
\[
 \ell_t=[0:1:t],\qquad
 \eta\kappa^{-1}\ell_t=[1:-t:0]
\]
meet at $(t^2:t:-1)$. This point belongs to the line
$[1:0:\varphi]=\kappa[0:\varphi:1]$ of $\cO_6$ exactly when
\[
 t^2-\varphi=0.
\]
Thus this concurrence selects $t=\pm s$. At $t=s$, the orbit $I\cdot\ell_s$ is $\cO_{30}$, and the intersection point is $(\varphi:s:-1)$, whose orbit consists of the $60$ triple points with profile $(0,1,2)$ in Table~\ref{tab:points}. The complete incidence calculation in Section~\ref{sec:certificate} proves simpliciality.

\begin{remark}\label{rem:field}
The quartic extension in the construction cannot be removed by a real projective change of coordinates to $\Q(\sqrt5)$. With the convention $\operatorname{CR}(0,\infty;c,d)=c/d$, four lines of $\A_{51}$ through $P$ have cross-ratio
\[
 \operatorname{CR}(0,\infty;s,\varphi^{-1})=\varphi s=s^3.
\]
This number does not lie in $\Q(\sqrt5)$: otherwise $s=s^3/\varphi$ would lie there, but $\varphi$ is not a square in $\Q(\sqrt5)$, since its norm is $-1$. Cross-ratios of concurrent lines are projectively invariant. Consequently $\A_{51}$ has no projectively equivalent realization over $\Q(\sqrt5)$. This statement concerns the given projective realisation, without asserting rigidity of its intersection lattice.
\end{remark}

\section{Exact verification}\label{sec:verification}

The accompanying file \texttt{verify\_a51.py} is a standalone verifier requiring only Python~3.9 or later and its standard library. It constructs $I$, generates the three line orbits, and checks agreement with the closed formulas~\eqref{eq:15}--\eqref{eq:30}. Independently of the point table, it forms the cross product of every pair of line normals, identifies proportional point coordinates, and tests every point against all $51$ lines. It then verifies the twelve point orbits, the pair identity~\eqref{eq:pairs}, and the numerical invariants in Theorem~\ref{thm:main}. It also checks that each of the $21$ original lines contains $20$ vertices and each of the $30$ added lines contains $22$. Every field element is represented as $a_0+a_1s+a_2s^2+a_3s^3$, with $a_i\in\Q$; arithmetic and zero tests are exact. The chosen real embedding is specified by $s=\sqrt{(1+\sqrt5)/2}>0$. No tolerance or numerical root recognition enters the incidence computation. The coordinate data and an exported incidence certificate are supplied as ancillary files. The basic check is run by
\begin{verbatim}
python3 verify_a51.py
\end{verbatim}
The option \texttt{--export-json certificate.json} writes all line normals, intersection points, and their incidences in this field representation. These files make the finite calculations in the proof reproducible without relying on a separate arrangement package.

\section*{Acknowledgments}

The author acknowledges support from the São Paulo Research Foundation (FAPESP), grant 2022/09891-4.

\section*{Data and code availability}

The data and code supporting the computations in this article
are available at: 

\href{https://github.com/TomasSilva/A51-simplicial-arrangement}{https://github.com/TomasSilva/A51-simplicial-arrangement}

\section*{AI usage}
Claude (Anthropic) assisted with the computations and the first draft of this article. 
Codex (OpenAI) assisted with revising the exposition, organisation, bibliography and \LaTeX{} source, and with checking compilation and layout.

\bibliographystyle{plain}
\bibliography{refs}
\end{document}